\documentclass[11pt,oneside,reqno]{amsart}%
\usepackage{amsfonts}
\usepackage{amsmath}
\usepackage{amssymb}
\usepackage{graphicx}
\usepackage{epstopdf} 
\usepackage{version}
\usepackage{bm}
\usepackage{enumerate}
\usepackage{epsfig}
\usepackage{caption}
\usepackage{setspace}
\usepackage{caption}
\DeclareCaptionLabelFormat{cont}{#1~#2\alph{ContinuedFloat}}
\DeclareMathOperator*{\argmin}{arg\,min}
\theoremstyle{plain}
\newtheorem{theorem}{Theorem}[section]

\newtheorem{proposition}[theorem]{Proposition}

\newtheorem{assumption}[theorem]{Assumption}
\newtheorem{remark}[theorem]{Remark}

\newtheorem*{notation*}{Notation}
\numberwithin{equation}{section}
\excludeversion{comment1}

\begin{document}
\title{Reconstruction of piecewise-smooth multivariate functions from Fourier data}
\author{David Levin}

\footnotetext[1]{David Levin, School of Mathematical Sciences, Tel-Aviv University, Israel}

\maketitle

In some applications, one is interested in reconstructing a function $f$ from its Fourier series coefficients. The problem is that the Fourier series is slowly convergent if the function is non-periodic, or is non-smooth. In this paper, we suggest a method for deriving high order approximation to $f$ using a Pad\'e-like method. Namely, by fitting some Fourier coefficients of the approximant to the given Fourier coefficients of $f$.
Given the Fourier series coefficients of a function on a rectangular domain in $\mathbb{R}^d$, assuming the function is piecewise smooth, we approximate the function by piecewise high order spline functions. First, the singularity structure of the function is identified. For example in the 2-D case, we find high accuracy approximation to the curves separating between smooth segments of $f$. Secondly, simultaneously we find the approximations of all the different segments of $f$. We start by developing and demonstrating a high accuracy algorithm for the 1-D case, and we use this algorithm to step up to the multidimensional case.

\section{Introduction}

\medskip
\medskip
Fourier series expansion is a useful tool for representing and approximating functions, with application in many areas of applied mathematics. The quality of the approximation depends on the smoothness of the approximated function and on whether or not it is periodic.
For functions that are not periodic, the convergence rate is slow near the boundaries and the approximation by partial sums exhibits the Gibbs phenomenon.
There are several approaches that have been used to improve the convergence rate, mostly for the one-dimensional case. One approach is to filter out the oscillations, as discussed in several papers as \cite{GottliebShu} and \cite{Tadmor}. Another useful approach is to transform the Fourier series into an expansion in a different basis.
For the univariate case this approach is shown to be very efficient, as shown in \cite{GottliebShu} using Gegenbauer polynomials with suitably chosen parameters.
Further improvement of this approach is presented in \cite{GelbTanner} using Freud polynomials, achieving very good results for univariate functions with singularities.

An algebraic approach for reconstructing a piecewise
smooth univariate function from its first N Fourier coefficients
has been realized by Eckhoff in a series of papers \cite{Eckhoff1}, \cite{Eckhoff2}, \cite{Eckhoff3}. There
the “jumps” are determined by a corresponding system of linear equations. A full analysis of this approach is presented by Betankov \cite{Betankov}.
Nersessian and Poghosyan \cite{Nersessian} have used a rational Pad\'e type approximation strategy for approximating univariate non-periodic smooth functions.
For multiple Fourier series of smooth non-periodic functions, a convergence acceleration approach has been suggested by Levin and Sidi \cite{LevinSidi2001}.
More challenging is the case of multivariate functions with discontinuities, i.e., functions which are piecewise smooth. Here again, the convergence rate is slow, and near the discontinuities the approximation exhibits the Gibbs phenomenon. In this paper, we present a Pad\'e-like approach consisting of finding a piecewise-defined spline whose Fourier coefficients match the given Fourier coefficients. 

The main contribution of this paper is demonstrating that this approach can be successfully applied to the multivariate case. Namely, we present a strategy for approximating both non-periodic and non-smooth multivariate functions. We derive the numerical procedures involved and provide some interesting numerical results. We start by developing and demonstrating a high accuracy algorithm for the 1-D case, and use this algorithm to step up to the multidimensional case.

\section{The 1-D case}

In this section, we present the main tools for function approximation using its Fourier series coefficients.
We define the basis functions and describe the fitting strategy and develop the computation algorithm. After dealing with the smooth case we move on to approximating a piecewise smooth function with a jump singularity.

\subsection{Reconstructing smooth non-periodic functions}\hfill

\medskip
Let $f\in C^m[0,1]$, and assume we know the Fouries series expansion of $f$
\begin{equation}\label{FS1D}
f(x)=\sum_{n\in\mathbb{Z}}\hat{f}_ne^{2\pi inx}.
\end{equation}
The series converge pointwise for any $x\in [0,1]$, however, if $f$ is not periodic the convergence may be slow, and if $f(1)\ne f(0)$ the convergence is not uniform and the Gibbs phenomenon occurs near $0$ and near $1$. As discussed in \cite{Sidi1995} and \cite{LevinSidi2001}, one can apply convergence acceleration techniques for improving the convergence rate of the series. Another convergence acceleration approach is suggested by Gottlieb and Shu
\cite{GottliebShu} using Gegenbauer polynomials. Yet, in both approaches, the convergence rate is not much improved near $0$ and near $1$.  We suggest an approach in the spirit of Pad\'e approximation. A Pad\'e approximant is a rational function whose power series agrees as much as possible with the given power series of $f$. Here we look for approximations to $f$ whose Fourier coefficients agree with a subset of the given Fourier coefficients of $f$. The approximation space can be any favourable linear approximation space, such as polynomials or trigonometric functions.

We choose to build the approximation using $k$th order spline functions, represented in the B-spline basis:
\begin{equation}\label{sumaiB}
S^{[k]}_d(x)=\sum_{i=1}^{N_d}a_iB^{[k]}_d(x-id).
\end{equation}
$B^{[k]}_d(x)$ is the B-spline of order k with equidistant knots $\{-kd,...,-2d,-d,0\}$, and $N_d=1/d+k-1$ is the number of B-splines whose shifts do not vanish in $[0,1]$. The advantage of using spline functions is threefold: 

\begin{itemize}
\item The locality of the B-spline basis functions. 
\item A closed form formula for their Fourier series coefficients. 
\item Their approximation power, i.e., if $f\in C^k[0,1]$, there exists a spline $S^{[k]}_d$ such that $\|f-S^{[k]}_d\|_{\infty, [0,1]}\le Cd^k$.
\end{itemize}
The B-splines basis functions used in the 1D case are shown in Figure \ref{B12in01}.
\begin{figure}[!ht]
    \includegraphics[width=4in]{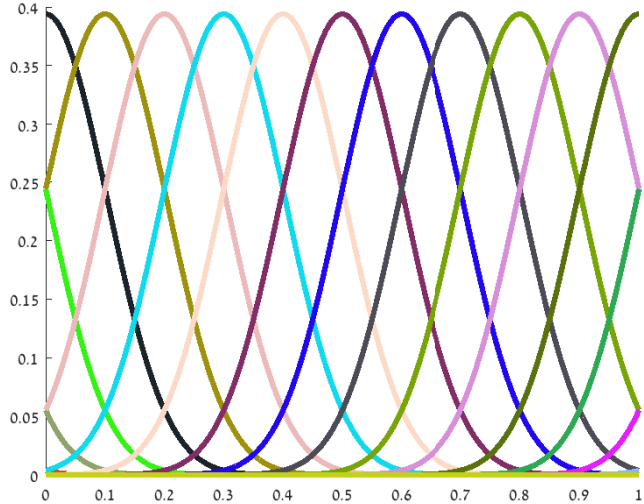}
    \caption{The B-splines used in Example 1}
    \label{B12in01}
\end{figure}
We denote by $S\equiv S^{[k]}_d|_{[0,1]}$ the restriction of $S^{[k]}_d$ to the interval $[0,1]$.
We find the coefficients $\{a_i\}_{i=1}^{N_d}$ by least-squares fitting, matching the first $M+1$ Fourier coefficients of $S$ to the corresponding $M+1$ Fourier coefficients of $f$. That is,
\begin{equation}\label{LS1}
\{a_i\}_{i=1}^{N_d}=\argmin\sum_{n=0}^M |\hat{f}_n-\hat{S}_n|^2.
\end{equation}
Notice that it is enough to consider the Fourier coefficients with non-negative indeices. 

We denote by $B_i\equiv B^{[k]}_d(\cdot-id)|_{[0,1]}$ the restriction of $B^{[k]}_d(\cdot-id)$ to the interval $[0,1]$, and by $\{\hat{B}_{i,n}\}$ its Fourier coefficients.
The induced system of linear equations for $a=\{a_i\}_{i=1}^{N_d}$ is $Aa=b$ where
\begin{equation}\label{Aij}
A_{i,j}=\sum_{n=0}^M[Re(\hat{B}_{i,n})Re(\hat{B}_{j,n})+Im(\hat{B}_{i,n})Im(\hat{B}_{j,n})],\ \ 1\le i,j \le N_d,
\end{equation}
and
\begin{equation}\label{bi}
b_i=\sum_{n=0}^M[Re(\hat{B}_{i,n})Re(\hat{f}_{n})+Im(\hat{B}_{i,n})Im(\hat{f}_{n})],\ \ 1\le i \le N_d.
\end{equation}

\subsubsection{Numerical Example - The smooth 1-D case}\hfill

\medskip
We consider the test function $f(x)=xexp(x)+sin(8x)$, assuming only its Fourier series coefficients are given. We have used only the 20 Fourier coefficients $\{\hat{f}_n\}_{n=0}^{19}$, and computed an approximation using 12th degree splines with equidistant knots' distance $d=0.1$.
 For this case, the matrix $A$ is of size $19\times 19$, and $cond(A)=5.75\times 10^{20}$. We have employed an iterative refinement algorithm described below to obtain a high precision solution. The results are shown in the following two figures. In Figure \ref{F1example} we see the test function on the left and the approximation error on the right. Figure \ref{Fig1Dsmooth} shows the graph of $Log_{10}(\hat{f}_n)$ in blue and the graph of $Log_{10}(\hat{f}_n-\hat{S}_n)$. Notice the matching in the first Fourier coefficients reflected in the begining of the red graph.
 
\begin{remark}{\bf Iterative refinement}
The powerful iterative refinement method described in\cite{Wilk}, \cite{Moler} is as follows:

For solving a system $Ax=b$, we use some solver, e.g. the matlab pinv function. We obtain the solution $x^{(0)}=pinv(A)b$. Next we compute the residual $r^{(0)}=b-Ax^{(0)}$. In case $cond(A)$ is very large, the residual will be large. Now we solve again the system with $r^{(0)}$ at the right hand side,
and use the solution to correct $x^{(0)}$, to obtain
$$x^{(1)}=x^{(0)}+pinv(A)r^{(0)}.$$
We repeat this correction steps a few times, i.e., $r^{(k)}=b-Ax^{(k)}$, and
$$x^{(k+1)}=x^{(k)}+pinv(A)r^{(k)},$$
until the resulting residual $r^{(k)}$ is small enough.
\end{remark}

\begin{figure}[!ht]
    \includegraphics[width=6in]{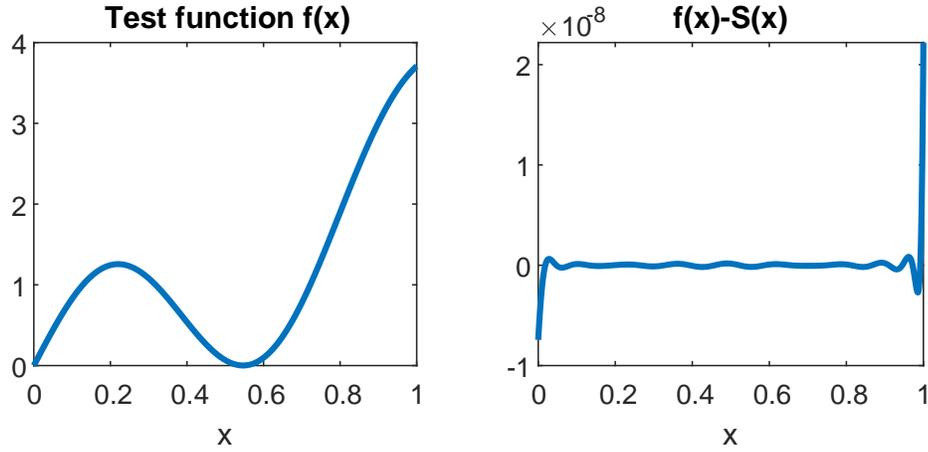}
    \caption{The test function (left) and the spline approximation error (right).}
    \label{F1example}
\end{figure}

\begin{figure}[!ht]
    \includegraphics[width=4in]{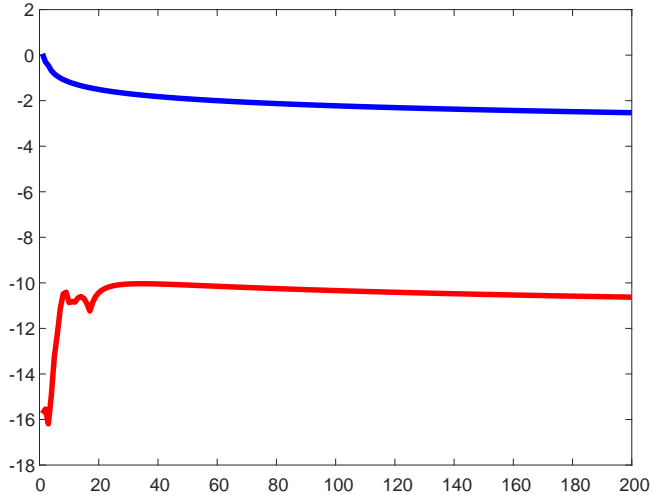}
    \caption{Eight orders of magnitude reduction in the Fourier coefficients.}
    \label{Fig1Dsmooth}
\end{figure}

 \subsection{Reconstructing non-smooth univariate functions}\hfill

\medskip
Let $f$ be a piecewise smooth function on $[0,1]$, defined by combined two pieces $f_1\in C^m[0,s^*]$ and $f_2\in C^m(s^*,1]$.
\begin{equation}\label{testf1Dnonsmooth0}
f(x)=
\begin{cases} 
f_1(x)& x\ge s^*,\\
f_2(x)& x< s^*. \\ 
\end{cases}
\end{equation}
 
Here again, we assume that all we know about $f$ is its Fourier series expansion.  In particular, we do not know the position $s^*\in [0,1]$ of the singularity of $f$. As in the case of a non-periodic function, the existence of a singularity in $[0,1]$ significantly influences the Fourier series coefficients and implies their slow decay. As we demonstrate below, good matching of the Fourier coefficients requires a good approximation of the singularity location. The approach we suggest here involves finding approximations to $f_1$ and $f_2$ simultaneously with a high precision identification of $s^*$.

Let $s$ be an approximation of the singularity location $s^*$, and let us follow the algorithm suggested above for the smooth case. The difference here is that now we look for two separate spline approximations:
\begin{equation}\label{S1}
S_1\equiv S^{[k]}_d|_{[0,s]}(x)=\sum_{i=1}^{N_d}a_{1i}B^{[k]}_d(x-id)|_{[0,s]}\sim f_1,
\end{equation}
and
\begin{equation}\label{S2}
S_2\equiv S^{[k]}_d|_{(s,1]}(x)=\sum_{i=1}^{N_d}a_{2i}B^{[k]}_d(x-id)|_{(s,1]}\sim f_2.
\end{equation}
The combination $S$ of $S_1$ and $S_2$ constitutes the aproximation to $f$. Here again we aim at matching the first $M+1$ Fourier coefficients of $f$ and of $S$. Here $S$ dependes upon the $N_d$ coefficients $\{a_{1i}\}$ of $S_1$, the $N_d$ coefficients $\{a_{2i}\}$ of $S_2$ and upon $s$. Therefore, the minimization process solves for all these unknowns:

\begin{equation}\label{LS2}
\big[\{a_{1i}\}_{i=1}^{N_d},\{a_{2i}\}_{i=1}^{N_d},s\big]=\argmin\sum_{n=0}^M |\hat{f}_n-\hat{S}_n|^2.
\end{equation}
The minimization is non-linear w.r.t. $s$, and linear w.r.t. the other unknowns. Therefore, the minimization problem is actually a one parameter non-linear minimization problem, the parameter $s$. Using the approximation power of $k$th order splines ($k\le m$), and considering the value of the objective cost function for $s=s^*$, we can deduce that the minimal value of $\sum_{n=0}^M |\hat{f}_n-\hat{S}_n|^2$ is $O(d^{2k})$. We also observe that an $\epsilon$ deviation from $s^*$ implies a bounded deviation of the minimizing Fourier coeficients
\begin{equation}\label{epsdev}
\max_{n\in\mathbb{Z}}|\hat{f}_n-\hat{S}_n|\le c_1\epsilon+c_2d^k.
\end{equation}
As shown below, these observations can be used for finding a good approximation to $s^*$.

We denote by $B_{1i}\equiv B^{[k]}_d(\cdot-id)|_{[0,s]}$ the restriction of $B^{[k]}_d(\cdot-id)$ to the interval $[0,s]$, and by  $B_{2i}\equiv B^{[k]}_d(\cdot-id)|_{(s,1]}$ the restriction of $B^{[k]}_d(\cdot-id)$ to the interval $(s,1]$. We concatenate these two sequences of basis functions, $\{B_{1i}\}$ and $\{B_{2i}\}$ into one sequence $\{B_i\}_{i=1}^{2N_d}$, and denote their Fourier coefficients by $\{\hat{B}_{i,n}\}_{n\in\mathbb{Z}}.$
For a given $s$, the induced system of linear equations for the splines' coefficients $a=(\{a_{1i}\}_{i=1}^{N_d},\{a_{2i}\}_{i=1}^{N_d})$ is $Aa=b$ defined as follows:

\begin{equation}\label{Aijns}
A_{i,j}=\sum_{n=0}^M[Re(\hat{B}_{i,n})Re(\hat{B}_{j,n})+Im(\hat{B}_{i,n})Im(\hat{B}_{j,n})],\ \ 1\le i,j \le 2N_d,
\end{equation}
and
\begin{equation}\label{bins}
b_i=\sum_{n=0}^M[Re(\hat{B}_{i,n})Re(\hat{f}_{n})+Im(\hat{B}_{i,n})Im(\hat{f}_{n})],\ \ 1\le i \le 2N_d.
\end{equation}

\begin{remark}
Due to the locality of the B-splines, some of the basis functions $\{B_{1i}\}$ and $\{B_{2i}\}$ may be identical $0$. It thus seems better to use only the non-zero basis functions. From our experience, since we use the generalized inverse approach for solving the system of equations, using all the basis functions gives the same solution.

The generalized inverse approach computes the least-squares solution to a system of linear equations that lacks a unique solution. It is also called the {\bf Moore-Penrose inverse}, and is computed by matlab {\bf pinv}.
\end{remark}

*** The above construction can be carried out to the case of several singular points.

\subsubsection{ Finding $s^*$.}\label{Findingsstar}\hfill

\medskip
We present the strategy for finding $s^*$ together with a specific numerical example.
We consider a test function on $[0,1]$ with a jump discontinuity at  $s^*=0.5$:
\begin{equation}\label{testf1Dnonsmooth}
f(x)=
\begin{cases} 
f_1(x)=sin(5x)& x\ge s^*,\\
f_2(x)=\frac{1}{(x-0.5)^2+0.5} & x< s^*. \\ 
\end{cases}
\end{equation}  
As expected, the Fourier series of $f$ is slowly convergent, and it exhibits the Gibbs phenomenon near the ends of $[0,1]$ and near $s^*$. In Figure \ref{GibbsandDiff200}, on the left, we present the sum of the first 200 terms of the Fourier series, computed at 20000 points in $[0,1]$. This sum is nonacceptable as an approximation to $f$, and yet we can use it to obtain a good initial approximation to $s_0\sim s^*$. On the right graph, we plot the first differences of the values in the left graph. The maximal difference is achieved at a distance of order $10^{-4}$ from $s^*$. 

Having a good approximation $s_0\sim s^*$ is not enough for achieving a good approximation to $f$. However, $s_0$ can be used as a starting point for an iterative method leading to a high precision approximation to $s^*$.
To support this assertion we present the graph in Figure \ref{nearsstar}, depicting the maximum norm of the difference between 1000 of the given Fourier coefficients and the corresponding Fourier coefficients of the approximation $S$, as a function of $s$, near $s^*=0.5$. This function is almost linear on each side of $s^*$, and simple quasi-Newton iterations converge very fast to $s^*$.

After obtaining a high accuracy approximation to $s^*$, we use it for deriving the piecewise spline approximation to $f$. Figure \ref{error1Dnonsmooth} depicts the approximation error, while Figure \ref{nonsmooth1Dfferror} shows $Log_{10}$ of the absolute values of the given Fourier coefficients of $f$ (in blue), and the corresponding values for the Fourier coefficients of $f-S$ (in red). The graph shows a reduction of $\sim 7$ orders of magnitude.


\begin{figure}[!ht]
    \includegraphics[width=6in,height=2.5in]{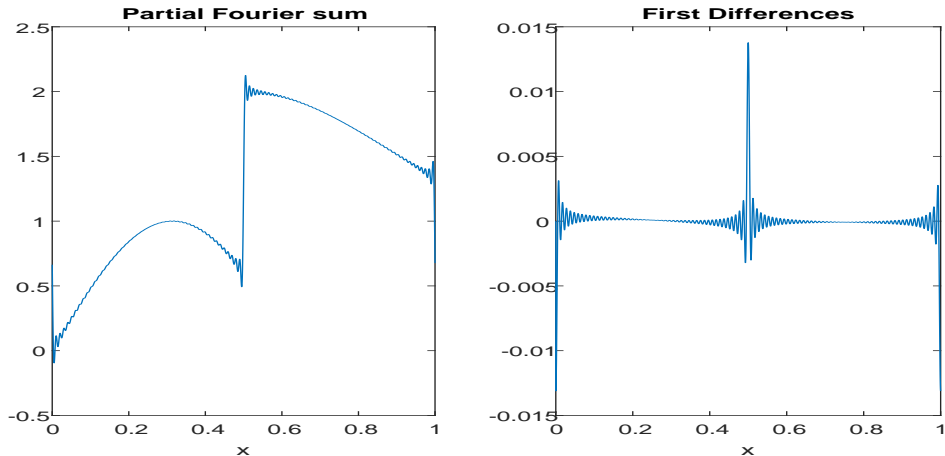}
    \caption{A partial Fourier sum (left) and its first differences (right).}
    \label{GibbsandDiff200}
\end{figure}

\begin{figure}[!ht]
    \includegraphics[width=4in,height=2.5in]{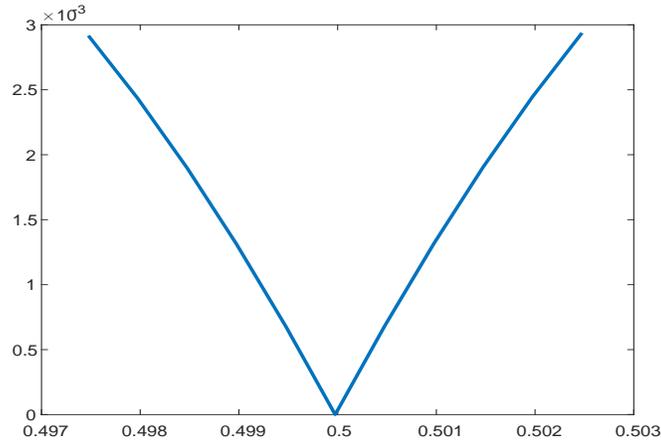}
    \caption{The graph of the error $\|\hat{f}-\hat{S}\|$ as a function of $s$ near $s^*=0.5$.}
    \label{nearsstar}
\end{figure}

\begin{figure}[!ht]
    \includegraphics[width=4in]{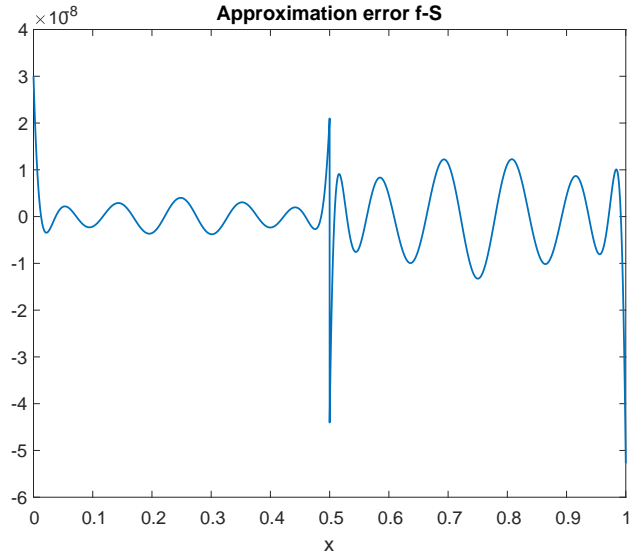}
    \caption{The approximation error for the 1D non-smooth case.}
    \label{error1Dnonsmooth}
\end{figure}

\begin{figure}[!ht]
    \includegraphics[width=4in]{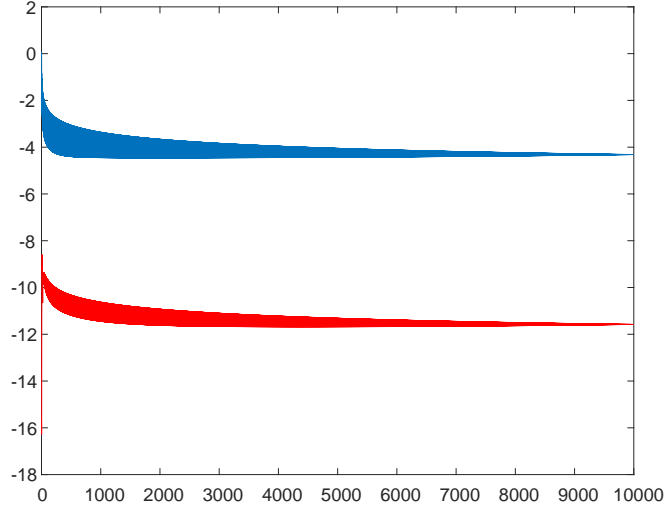}
    \caption{$Log_{10}$ of the given Fourier coefficients (blue), and of the Fourier coefficients of the approximation error (red)}
    \label{nonsmooth1Dfferror}
\end{figure}

\vfill\eject
\subsubsection{The 1-D approximation procedure}\hfill

\medskip
Let us sum up the suggested approximation procedure:

\begin{enumerate}
\item Choose the approximation space $\Pi$ for approximating $f_1$ and $f_2$.
\item Define the number of Fourier coefficients to be used for building the approximation such that
\begin{equation}\label{chooseM1}
M+1\ge 2dim(\Pi).
\end{equation}
\item Find first approximation to $s^*$: Compute a partial Fourier sum and locate maximal first order difference.
\item Calculate the first $M+1$ Fourier coefficients of the basis functions of $\Pi$, truncated at $s^*$.
\item Use the above Fourier coefficients to compute the approximation to $f_1$ and $f_2$ by solving the system of linear equation defined by (\ref{Aijns}), (\ref{bins}).
\item Update the approximation to $s^*$, by performing quasi-Newton iterations to reduce the objective function in (\ref{LS2}).
\item Go back to (4) to update the approximation.

\end{enumerate}

\section{The 2-D case - non periodic and non-smooth}

\subsection{The smooth 2-D case}\hfill

\medskip
Let $f\in C^m[0,1]^2$, and assume we know its Fouries series expansion
\begin{equation}\label{FS2D}
f(x,y)=\sum_{m\in\mathbb{Z}}\sum_{n\in\mathbb{Z}}\hat{f}_{mn}e^{2\pi imx}e^{2\pi iny}.
\end{equation}
Such series are obtained when solving PDE using spectral methods. However, if the function is not periodic, or, as in the case of hyperbolic equations, the function has a jump discontinuity along some curve in $[0,1]^2$, the convergence of the Fourier series is slow. Furthermore, the approximation of $f$ by its partial sums suffers from the Gibbs phenomenon near the boundaries and near the singularity curve.

We deal with the case of smooth non-periodic 2-D functions in the same manner as we did for the univariate case. We look for a bivariate spline function $S$ whose Fourier coefficients match the Fourier coefficients of $f$. As in the univariate case, it is enough to match the coefficients of low frequency terms in the Fourier series. The technical difference in the 2-D case is that we look for a tensor product spline approximation, using tensor product $k$th order B-spline basis functions. 
\begin{equation}\label{sumaijB}
S^{[k]}_d(x,y)=\sum_{i=1}^{N_d}\sum_{j=1}^{N_d}a_{ij}B^{[k]}_d(x-id)B^{[k]}_d(y-jd).
\end{equation}
The system of equations for the B-spline coefficients is the same as the system defined by (\ref{Aij})-(\ref{bi}) in the univariate case, only here we reshape the unknowns as a vector of $N_d^2$ unknowns.

\subsubsection{Numerical Example - The smooth 2-D case}\hfill

\medskip
We consider the test function 
$$f(x,y)=\frac{10}{1+10(x^2+(y-1)^2)}+sin(10(x-y)),$$ 
assuming only its Fourier series coefficients are given. We have used only $160$ Fourier coefficients, and constructed an approximation using 10th degree tensor product splines with equidistant knots' distance $d=0.1$ in each direction.
For this case, the matrix $A$ is of size $361\times 361$, and $cond(A)=6.2\times 10^{30}$. Again, we have employed the iterative refinement algorithm to obtain a high precision solution (relative error $10^{-15}$). Computation time $\sim 18$ seconds.

In figure  \ref{testf2Dsmooth} we plot the test function on $[0,1]^2$. Note that it has high derivatives near $(0,1)$.

\begin{figure}[!ht]
    \includegraphics[width=4in]{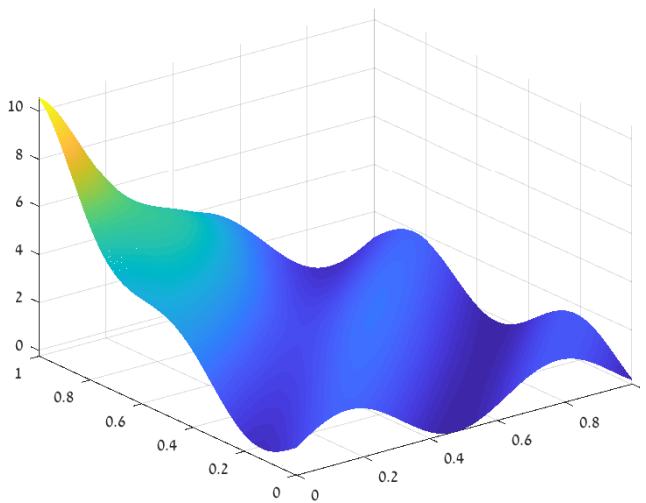}
    \caption{The test function for the smooth 2D case}
    \label{testf2Dsmooth}
\end{figure}

The approximation error $f-S^{[10]}_{0.1}$ is shown in figure \ref{error2Dsmooth}.
\begin{figure}[!ht]
    \includegraphics[width=4in]{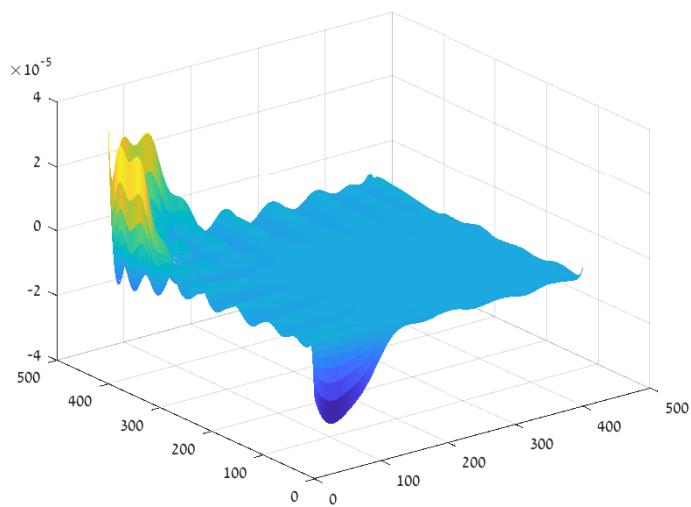}
    \caption{The approximation error $f-S^{[10]}_{0.1}$.}
    \label{error2Dsmooth}
\end{figure} 
To demonstrate the convergence acceleration of the Fourier series achieved by substracting the approximation from $f$, we present in Figure \ref{ffcoef2Dsmooth}
$Log_{10}$ of the absolute values of the Fourier coefficients of $f$ (in green) and the of the Fourier coefficients of $f-S^{[10]}_{0.1}$ (in blue), for frequencies $0\le m,n\le 200$. The magnitude of the Fourier coefficients is reduced by a factor of $10^5$, and even more so for the low frequencies due to the matching strategy used to derive the spline approximation.
\begin{figure}[!ht]
    \includegraphics[width=4in]{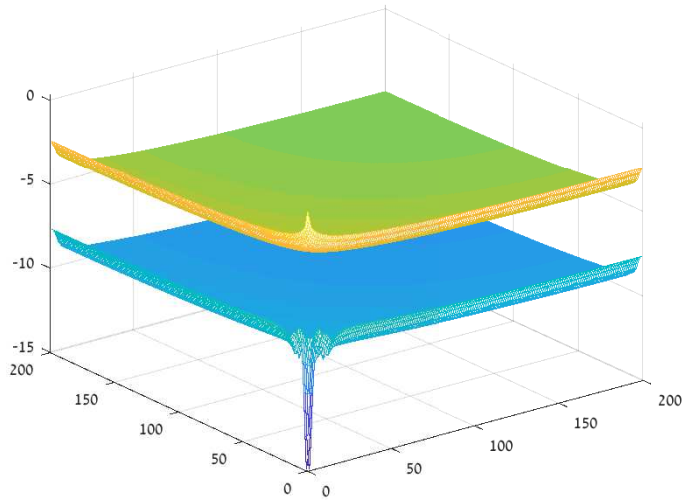}
    \caption{$Log_{10}$ of the Fourier coefficients before (green), and after (blue).}
    \label{ffcoef2Dsmooth}
\end{figure} 

\vfill\eject

\subsection{The non-smooth 2-D case}\label{Sec2Dns}\hfill

\medskip
Let $f$ be a piecewise smooth function on $[0,1]^2$, defined by combined two pieces $f_1\in C^m[\Omega_1]$ and $f_2\in C^m[\Omega_2]$, $\Omega_2=[0,1]^2\setminus \Omega_1$. Here again, we assume that all we know about $f$ is its Fourier series expansion.  In particular, we do not know the position of the dividing curve separating $\Omega_1$ and $\Omega_2$. We denote this curve by $\Gamma^*$, and we assume that it is a $C^m$-smooth curve. As in the case of a non-periodic function, the existence of a singularity curve in $[0,1]^2$ significantly influences the Fourier series coefficients and implies their slow decay. In case of discontinuity of $f$ across $\Gamma^*$, partial sums of the Fourier series exhibit the Gibbs phenomenon near $\Gamma^*$ . As demonstrated below, good matching of the Fourier coefficients requires a good approximation of the singularity location. As in the univariate non-smooth case, the computation algorithm involves finding approximations to $f_1$ and $f_2$ simultaneously with a high precision identification of $\Gamma^*$.

Evidently, finding a high precision approximation of the singularity curve $\Gamma^*$ is more involved than finding a high precision approximation to the singularity point $s^*$ in the univariate case. Let $D_{\Gamma^*}(x,y)$ be the signed-distance function corresponding to the curve $\Gamma^*$:
\begin{equation}\label{DGamma}
D_{\Gamma^*}(x,y)=
\begin{cases} 
\ \ dist((x,y),\Gamma^*)& (x,y) \in\Omega_1,\\
-dist((x,y),\Gamma^*)& (x,y) \in\Omega_2.\\ 
\end{cases}
\end{equation}
In looking for an approximation to $\Gamma^*$, we look for an approximation to $D_{\Gamma^*}$. Here again we are using a tensor product spline approximants, the same set of spline functions described in the previous section. Since the curve is $C^m$, it can be shown that one can construct a spline function  $\tilde{D}$ of order $k\le m$, with knots' distance $h$, which approximates $D_{\Gamma^*}$ near $\Gamma^*$ so that the Hausdorff distance belween the zero level set of $\tilde{D}$ and $\Gamma^*$ is $O(h^{k})$.

Let $D_{\bar{b}}$ be a spline approximation to $D_{\Gamma^*}$, with spline coefficients $\bar{b}=\{b_{ij}\}_{i,j=1}^{N_h}$:
\begin{equation}\label{Db}
D_{\bar{b}}(x,y)=\sum_{i=1}^{N_h}\sum_{j=1}^{N_h}b_{ij}B^{[k]}_h(x-ih)B^{[k]}_h(y-jh).
\end{equation}
For a given $D_{\bar{b}}$ we define the approximation to $f$
similar to the construction in the univariate case by equations (\ref{S1}), (\ref{S2}), (\ref{LS2}). We look here for an approximation $S$ to $f$ which is a combination of two bivariate splines components:
\begin{equation}\label{S12D}
S(x,y)=\sum_{i=1}^{N_d}\sum_{j=1}^{N_d}a_{1ij}B^{[k]}_d(x-id)B^{[k]}_d(y-jd), \ \ \ D_{\bar{b}}(x,y)\ge 0,
\end{equation}
\begin{equation}\label{S22D}
S(x,y)=\sum_{i=1}^{N_d}\sum_{j=1}^{N_d}a_{2ij}B^{[k]}_d(x-id)B^{[k]}_d(y-jd),
\ \ \ D_{\bar{b}}(x,y) < 0,
\end{equation}
such that $(2M+1)^2$ Fourier coefficients of $f$ and $S$  are matched in the least-squares sense:
\begin{equation}\label{LS2D}
\big[\{a_{1ij}\}_{i,j=1}^{N_d},\{a_{2ij}\}_{i,j=1}^{N_d},\{b_{ij}\}_{i,j=1}^{N_d}\big]=\argmin\big(\sum_{m,n=-M}^M |\hat{f}_{mn}-\hat{S}_{mn}|^2\big).
\end{equation}
We denote by $B_{1ij}(x,y)$ the restriction of $B^{[k]}_d(x-id)B^{[k]}(y-jd)$ to the domain defined by 
$D_{\bar{b}}(x,y)\ge 0$, and by $B_{2ij}(x,y)$ the restriction of $B^{[k]}_d(x-id)B^{[k]}(y-jd)$ to the domain defined by 
$D_{\bar{b}}(x,y)< 0$. We concatenate these two sequences of basis functions, $\{B_{1ij}\}$ and $\{B_{2ij}\}$ into one sequence $\{B_{ij}\}_{i=1,j=1}^{N_d,2N_d}$, denoting their Fourier coefficients by $\{\hat{B}_{ij,n}\}_{n\in\mathbb{Z}}$, and rearranging them (for each $n$) in vectors of lenght $2N_d^2$, $\{\hat{B}_{i,n}\}^{2N_d^2}_{i=1,n\in\mathbb{Z}}$.
For a given $D_{\bar\hat{b}}$, the induced system of linear equations for the splines' coefficients $a=(\{a_{1ij}\}_{i,j=1}^{N_d},\{a_{2ij}\}_{i,j=1}^{N_d})$ is $Aa=b$ defined as follows:

\begin{equation}\label{Aijns2D}
A_{i,j}=\sum_{m,n=-M}^M[Re(\hat{B}_{i,n})Re(\hat{B}_{j,n})+Im(\hat{B}_{i,n})Im(\hat{B}_{j,n})],\ \ 1\le i \le 2N_d^2,
\end{equation}
and
\begin{equation}\label{bins2D}
b_i=\sum_{m,n=-M}^M[Re(\hat{B}_{i,n})Re(\hat{f}_{n})+Im(\hat{B}_{i,n})Im(\hat{f}_{n})],\ \ 1\le i \le 2N_d^2.
\end{equation}

For a given choice of $\bar{b}=\{b_{ij}\}$, the coefficients 
$\{a_{1ij}\}_{i,j=1}^{N_d},\{a_{2ij}\}_{i,j=1}^{N_d}$ are obtained by solving a linear system of equations, and properly rearranging the solution. However, finding the optimal $\bar{b}$ is a non-linear problem that requires an iterative process and is much more expensive. 

\begin{remark}
Representing the singularity curve of the approximation $S$ as the zero level set of the bivariate spline function $D_{\bar{b}}$ is the way to achieve a smooth control over the approximation. As a result, the objective function in (\ref{LS2D}) varies smoothly with respect to the spline coefficients $\{b_{ij}\}$.
\end{remark}

Here again we choose to demonstrate the whole approximation procedure alongside a specific numerical example.

\subsubsection{The approximation procedure - a numerical example}\hfill

\medskip
Consider a piecewise smooth function on $[0,1]^2$ with a jump singularity across the curve $\Gamma^*$ which is the quarter circle defined by $x^2+y^2=0.5$. The test function is shown in Figure \ref{testf2Dnonsmooth} and is defined as
\begin{equation}\label{testf2Dnonsmooth}
f(x,y)=
\begin{cases} 
(x^2+y^2-0.5)sin(10(x+y))& x^2+y^2\ge 0.5,\\
(x^2+y^2-0.5)sin(10(x+y))+2x & x^2+y^2< 0.5.\\ 
\end{cases}
\end{equation}
\begin{figure}[!ht]
    \includegraphics[width=4in]{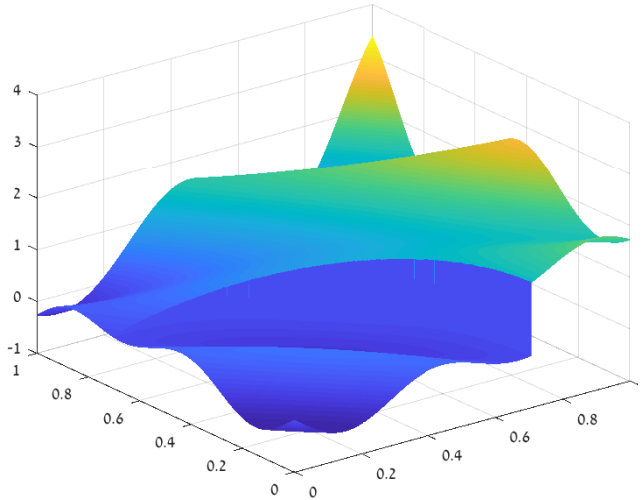}
    \caption{The test function for the 2D non-smooth case.}
    \label{testf2Dnonsmooth}
\end{figure}

In the univariate case, in Section \ref{Findingsstar}, we use the Gibbs phenomenon in order to find an initial approximation $s_0$ to the singularity location $s^*$. The same idea, with some modifications to the 2D case, is applied here. The truncated Fourier sum
\begin{equation}\label{F50}
f_{50}(x,y)=\sum_{m,n=-50}^{50}\hat{f}_{mn}e^{2\pi imx}e^{2\pi iny}.
\end{equation}
gives an approximation to $f$, but the approximation suffers from a Gibbs phenomenon near the boundaries of the domain and near the singularity curve $\Gamma^*$.  We evaluated $f_{50}$ on a $400\times 400$ mesh on $[0,1]^2$, and enhanced the Gibbs effect by applying first order differences along the $x$-direction. The results are depicted in Figure \ref{Gibbsnearsingularity}. The locations of large $x$-direction differences and of large $y$-direction differences within $[0,1]^2$ indicate the location of $\Gamma^*$.

\begin{figure}[!ht]
    \includegraphics[width=4in]{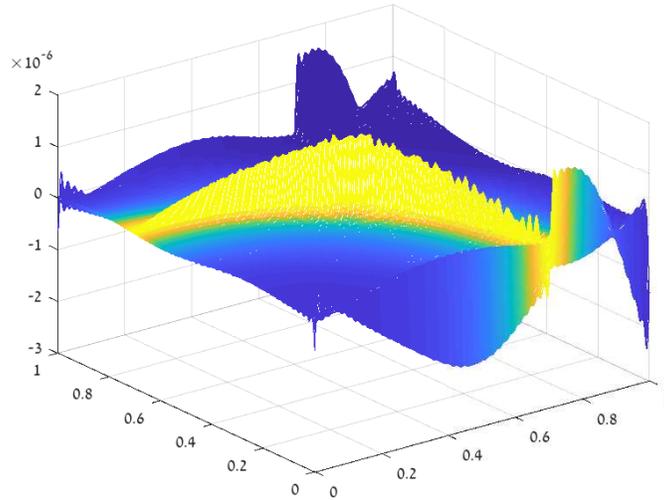}
    \caption{First order $x$-direction differences of a truncated Fourier sum - Notice the relatively high values at the boundary and near the singularity curve.}
    \label{Gibbsnearsingularity}
\end{figure}
\medskip
{\bf Building the initial approximation $D_{\bar{b}_0}$:}

\medskip
Searching along 50 horizontal lines ($x$-direction) for maximal $x$-direction differences, and along 50 vertical lines ($y$-direction) for maximal $y$ direction differences, we have found $72$ such maximum points, which we denote by $P_0$. We display these points (in red) in Figure \ref{Max1DiffPoints}, on top of the curve $\Gamma^*$ (in blue).
\begin{figure}[!ht]
    \includegraphics[width=4in]{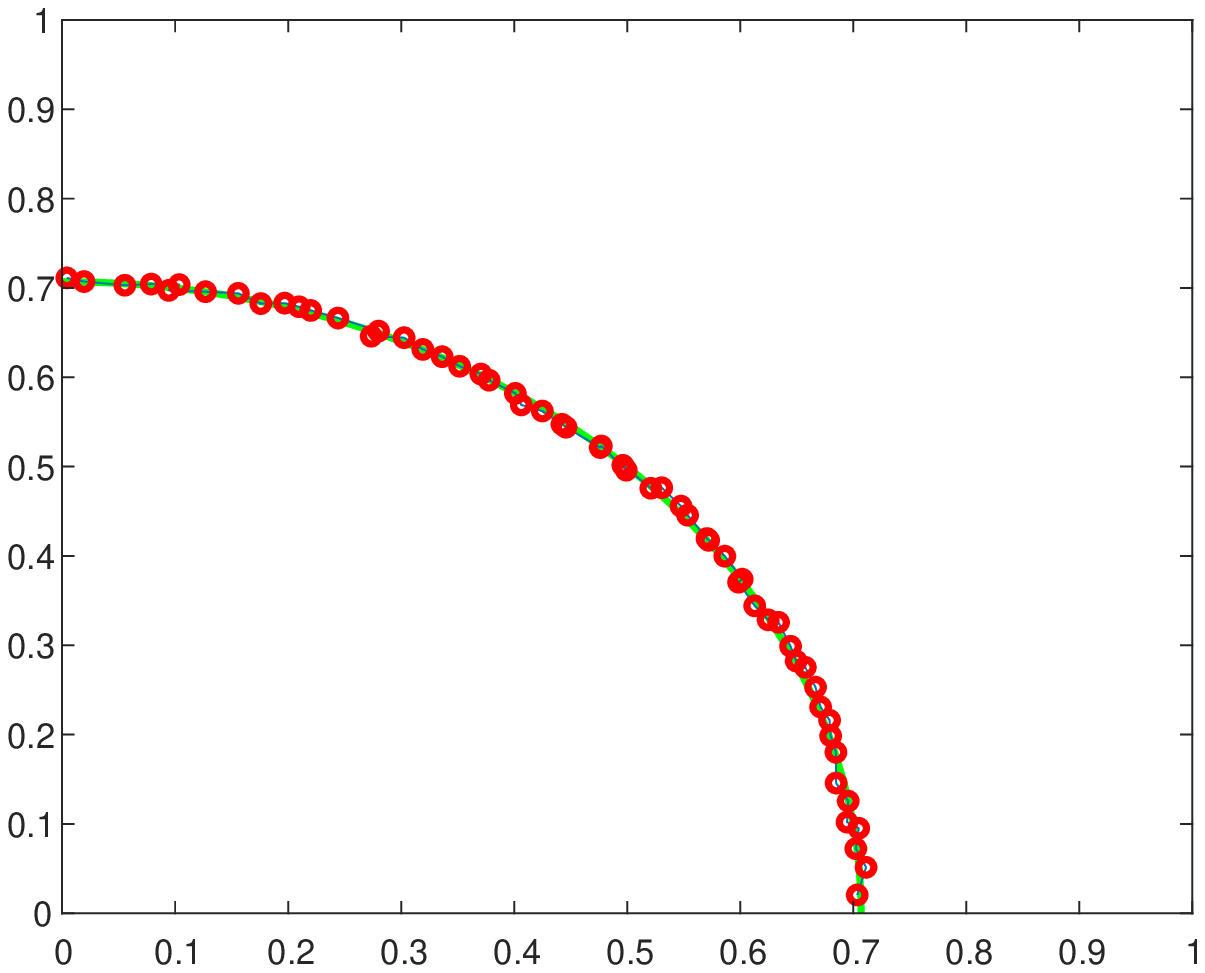}
    \caption{The singularity curve $\Gamma^*$ (blue) and points of maximal first differences of $f_{50}$.}
    \label{Max1DiffPoints}
\end{figure}
Now we use these points to construct the spline $D_{\bar{b}_0}$, whose zero level curve is taken as the initial approximation to $\Gamma^*$. To construct $D_{\bar{b}_0}$ we first overlay on $[0,1]^2$ a net of $11\times 11$ points, $Q_0$. These are the green points displayed in Figure \ref{SignedDistanceAppr4}.

\begin{figure}[!ht]
    \includegraphics[width=4in]{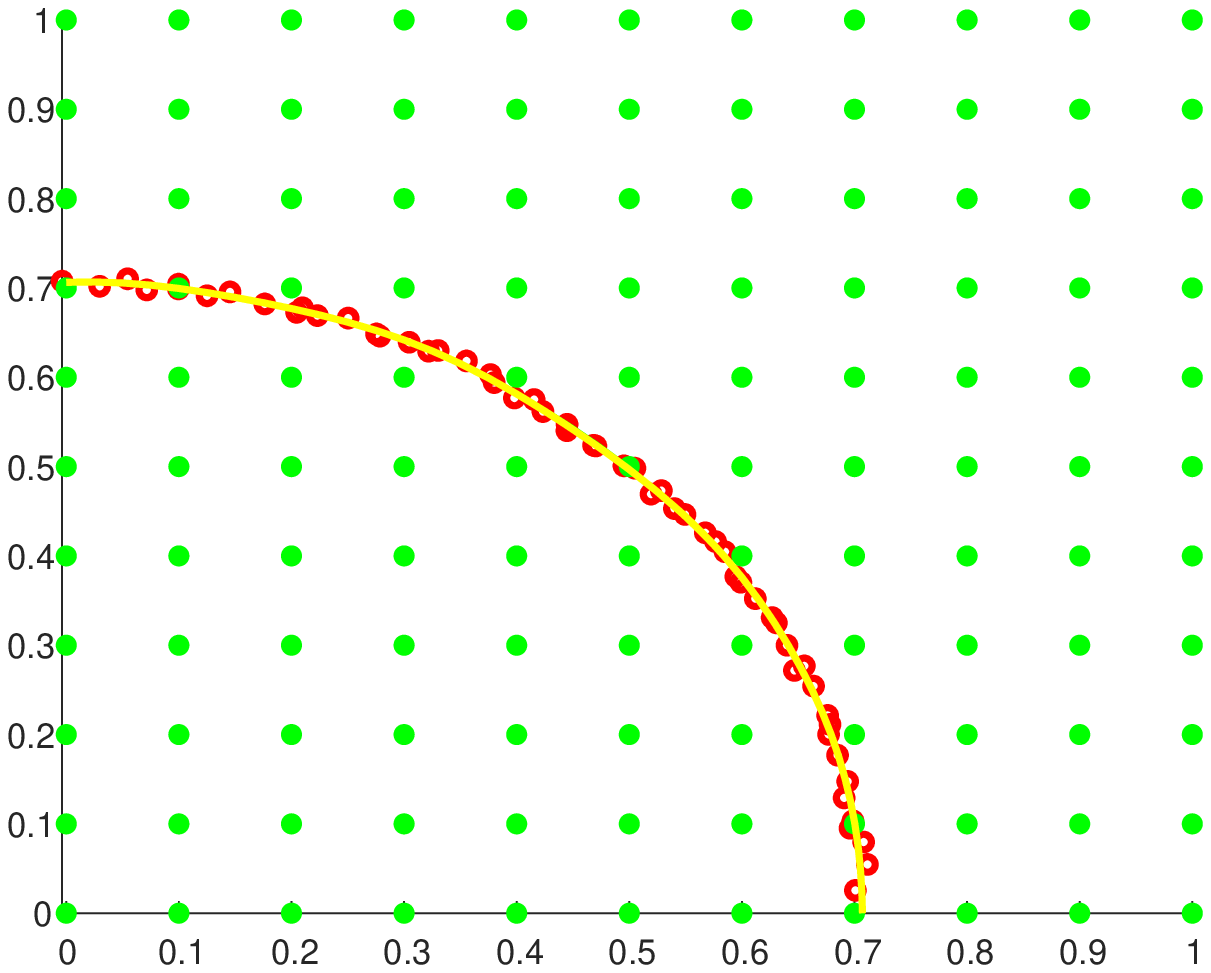}
    \caption{The singularity curve $\Gamma^*$ (blue) and points of maximal first differences of $f_{50}$.}
    \label{SignedDistanceAppr4}
\end{figure}

To each point in $Q_0$ we assign the value of its distance from the set $P_0$, with a plus sign for points which are on the right or above $P_0$, and a minus sign for the other points. To each point in $P_0$ we assign the value zero. The spline function $D_{\bar{b}_0}$ is now defined by the least-squares approximation to the values at all the points $P_0\cup Q_0$. We have used here tensor product splines of order $10$, on a uniform mesh with knots' distance $=0.1$. We denote the level curve zero of the resulting $D_{\bar{b}_0}$ as $\Gamma_0$, and this curve is depicted in yellow in Figure \ref{SignedDistanceAppr4}. It seems that 
$\Gamma_0$ is already a good approximation to $\Gamma^*$ (in blue), and thus it is a good starting point for achieving the minimization target (\ref{LS2D}).

\medskip
{\bf Improving the approximation to $\Gamma^*$, and building the  two approximants.}

\medskip
Starting from $D_{\bar{b}_0}$ we use a quasi-Newton method for iterative improvement of the approximation to $\Gamma^*$.
The expensive ingredient in the computation procedure is the need to recompute the Fourier coefficients of the $B$-splines for any new set of coefficients $\bar{b}$ of $D_{\bar{b}}$. We recall that we need $(2M+1)^2$ of these coefficients for each $B$-spline, and we have $2N_d^2$ $B$-splines. In the numerical example we have used $M=40$ and $N_d=19$. To illustrate the issue we present in Figure 
\ref{B10112Dnonsmooth} one of those $B$-spline whose support intersects the singularity curve. When the singularity curve is updated, the Fourier coefficients of this $B$-spline are recalculated.

\begin{remark}{\bf Calculating Fourier coefficients of the B-splines}
Calculating the Fourier coefficients of the B-splines is the most costly step in the approximation procedure. For the univariate case the Fourier coefficients of the B-splines can be computed analytically. For the smooth nultivariate case, with no singularity within the unit cube, piecewise Gauss quadrature may be used to compute the Fourier coefficients with high precision. The non-smooth multivariate case is more difficult, and more expensive. However, we noticed that using low precision approximations for the Fourier coefficients of the B-splines is fine. For example, in the above example, we have employed a simple numerical quadrature combined with fast Fourier transform, and we obtained the Fourier coefficients with a relative error     
$\sim 10^{-5}$. Yet the resulting approximation error is small $\|f-S\|_\infty<5\times 10^{-6}$, as seen in Figure \ref{Error2Dboptimalnoisy}. 
\end{remark}

\begin{figure}[!ht]
    \includegraphics[width=4in]{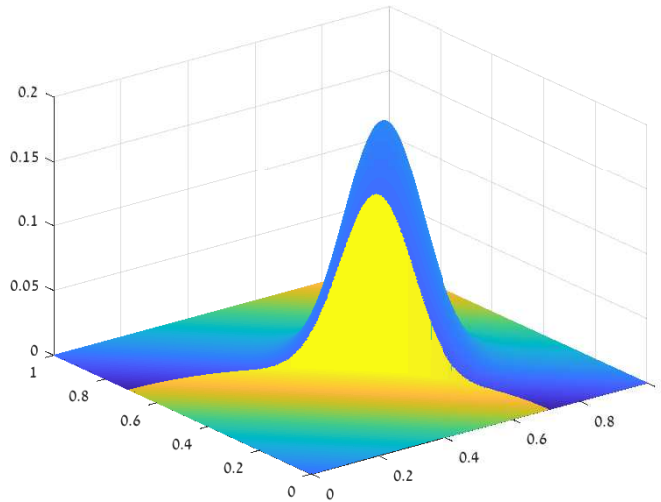}
    \caption{One of the tensor product B-splines used for the approximation of $f$, chopped off by the singularity curve.}
    \label{B10112Dnonsmooth}
\end{figure}

Using one quasi-Newton step we obtained new spline coefficients $\bar{b}_1$ and an improved approximation $\Gamma_1$ to $\Gamma^*$ as the zero level set of $D_{\bar{b}_1}$. Stopping the procedure at this point yield approximation results as shown in the figures below. Figure \ref{Error2Db0} shows the approximation error $f-S$ on $[0,1]^2\setminus U$, where $U$ is a small neighborhood of $\Gamma^*$. Figure \ref{Log10error2Db0} shows, in green, $Log_{10}$ of the magnitude of the giver Fourier coefficients $\hat{f}_{mn}$ and, in blue, $Log_{10}$ of the Fourier coefficients of the difference 
$f-S$. We observe a reduction of three orders of magnitude between the two.

\begin{figure}[!ht]
    \includegraphics[width=4in]{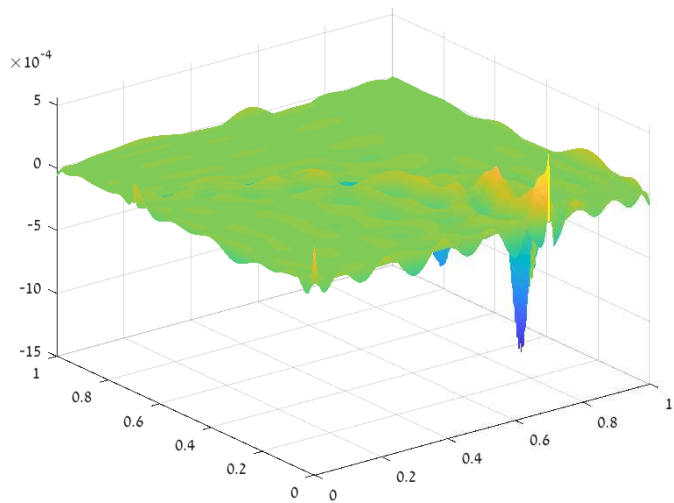}
    \caption{The approximation error with $D_{\bar{b}_1}$.}
    \label{Error2Db0}
\end{figure}

\begin{figure}[!ht]
    \includegraphics[width=4in]{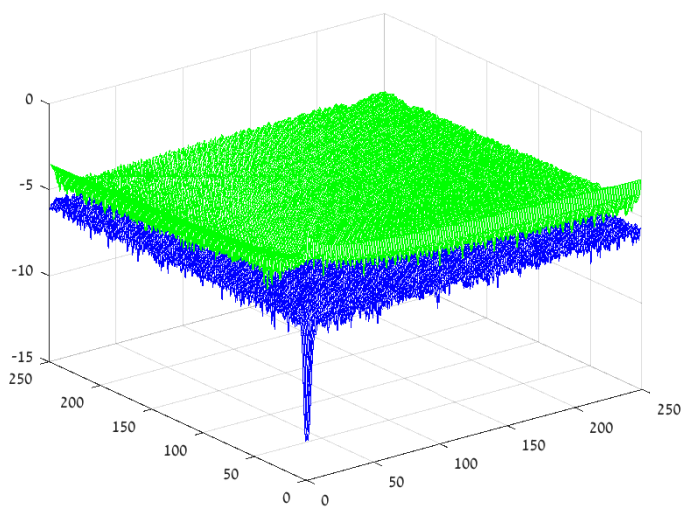}
    \caption{The magnitude reduction of the Fourier coefficients with $D_{\bar{b}_1}$.}
    \label{Log10error2Db0}
\end{figure}

Applying four quasi-Newton iterations took $\sim 24$ minutes execution time. The  
approximation of $\Gamma^*$ by  the zero level set of $D_{\bar{b}_4}$ is now with an error of $10^{-9}$. The consequent approximation error to $f$ is reduced as shown in Figure \ref{Error2Dboptimalnoisy}, and the Fourier coefficients of the error are reduced by 5 orders of magnitude, as shown in Figure \ref{Log10error2Dboptimalnoisy2}.

\begin{figure}[!ht]
    \includegraphics[width=4in]{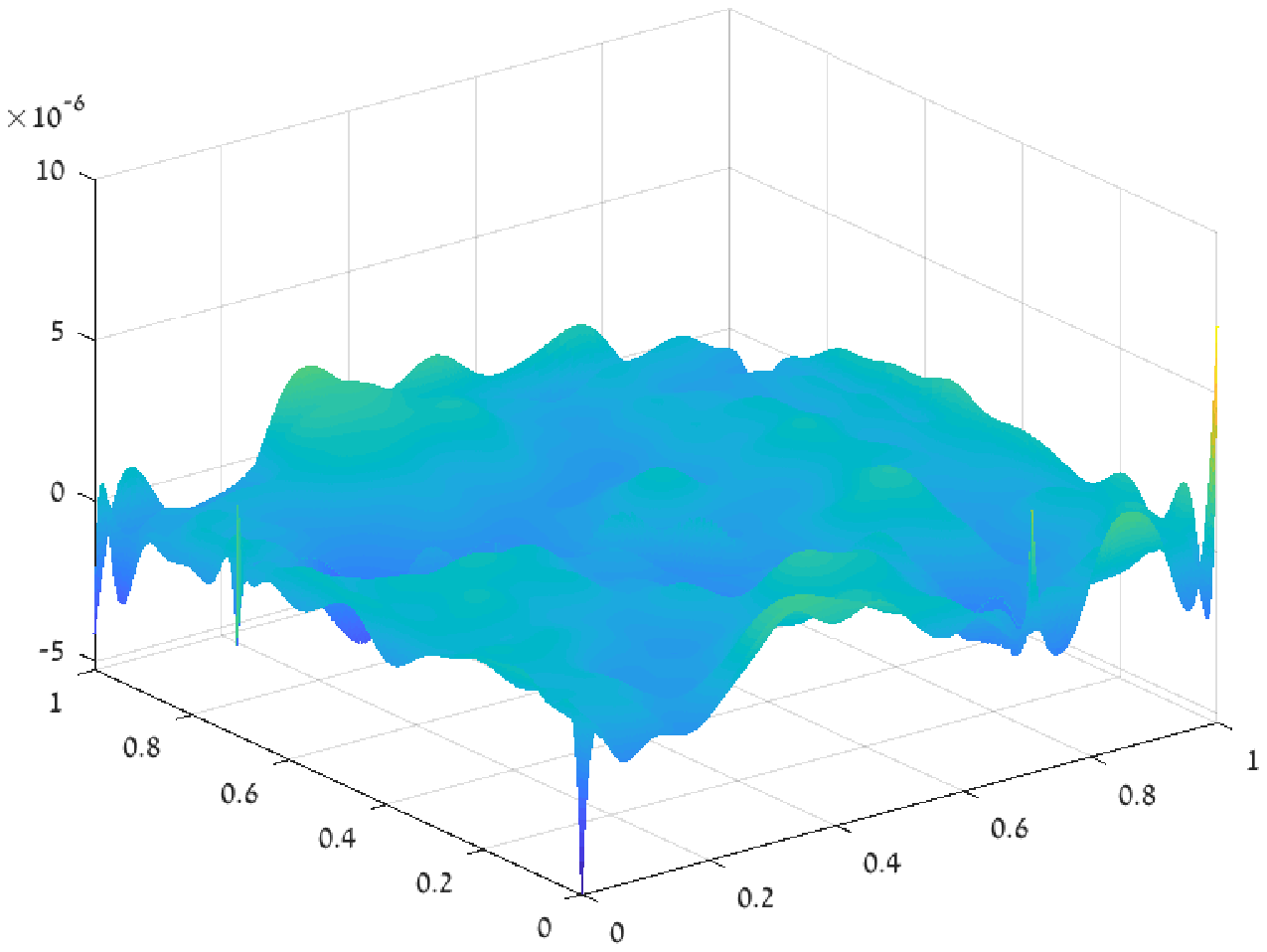}
    \caption{The approximation error with $D_{\bar{b}_4}$.}
    \label{Error2Dboptimalnoisy}
\end{figure}

\begin{figure}[!ht]
    \includegraphics[width=4in]{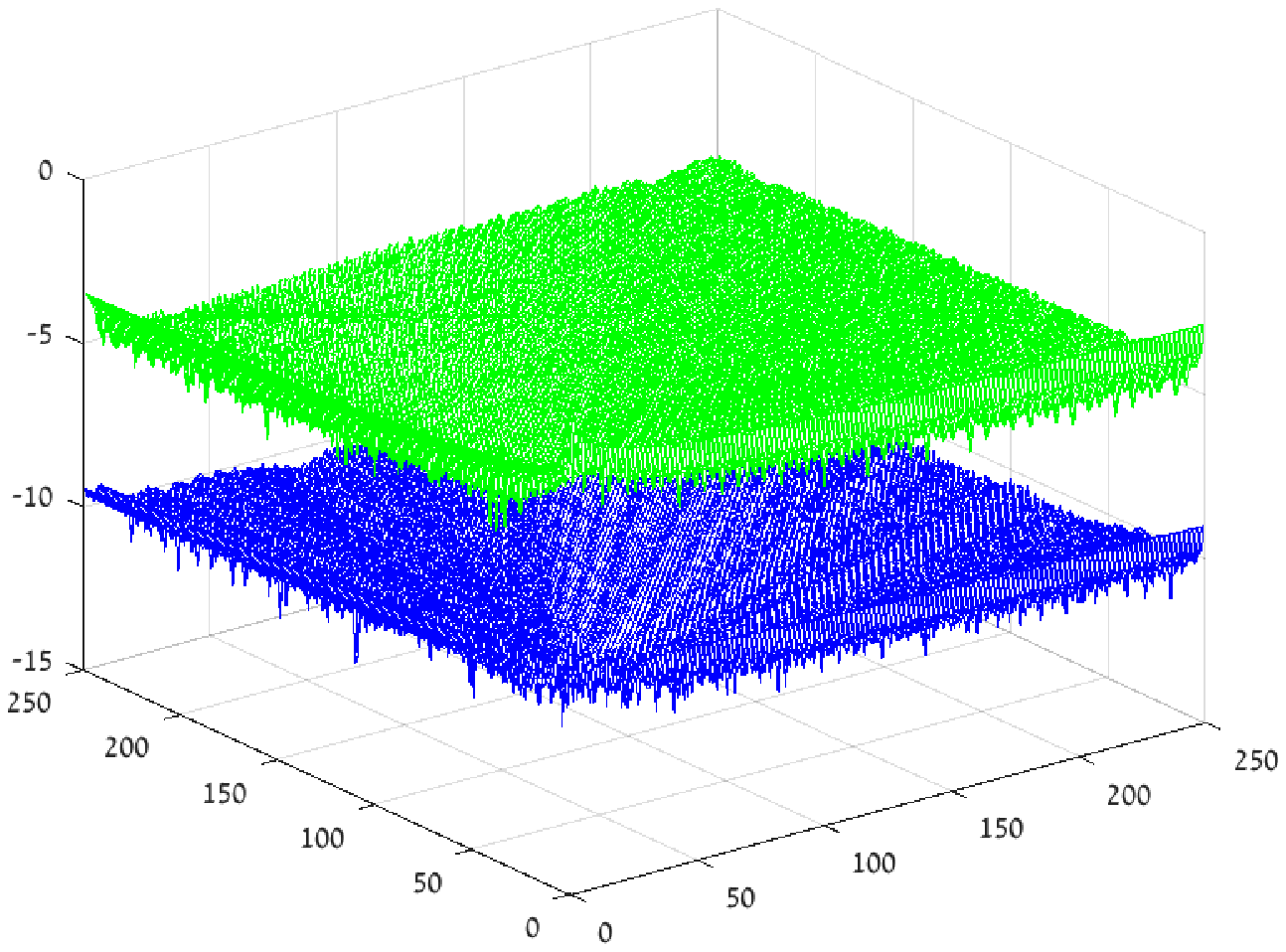}
    \caption{The magnitude reduction of the Fourier coefficients with $D_{\bar{b}_4}$.}
    \label{Log10error2Dboptimalnoisy2}
\end{figure}

\subsubsection{The 2-D approximation procedure}\label{appproc2D}\hfill

\medskip
Let us sum up the suggested approximation procedure:

\begin{enumerate}
\item Choose the approximation space $\Pi_1$ for approximating $f_1$ and $f_2$
and the approximation space $\Pi_2$ for approximating $\Gamma^*$.

\item Define the number of Fourier coefficients to be used for building the approximation such that
\begin{equation}\label{chooseM}
(2M+1)^2\ge 2dim(\Pi_1)+dim(\Pi_2).
\end{equation}
\item Find first approximation to $\Gamma*$:
\begin{enumerate}
\item Compute a partial Fourier sum and locate maximal first order differences along horizontal and vertical lines to find points $P_0$ near $\Gamma^*$, with assigned values $0$.
\item Overlay a net of points $Q_0$ as in Figure \ref{SignedDistanceAppr4}, with assigned signed-distance values.
\item Compute the least-squares approximation from $\Pi_2$ to the values at $P_0\cup Q_0$, denote it $D_{\bar{b}_0}$.
\end{enumerate}
\item Calculate the first $(2M+1)^2$ Fourier coefficients of the basis functions of $\Pi_1$, truncated w.r.t. the zero level curve of $D_{\bar{b}_0}$.
\item Use the above Fourier coefficients to compute the approximation to $f_1$ and $f_2$ by solving the system of linear equation defined by (\ref{Aijns2D}), (\ref{bins2D}).
\item Update $D_b$ to improve the approximation to $\Gamma^*$, by performing quasi-Newton iterations to reduce the objective function in (\ref{LS2D}).
\item Go back to (4) to update the approximation.

\end{enumerate}

\subsubsection{Lower order singularities}\hfill

\medskip
Let us assume that $f(x,y)$ is a continuous function, and that $f_x(x,y)$ is discontinuous across the singularity curve $\Gamma^*$. In this case we cannot use the Gibbs phnomenon effect to approximate the singularity curve. However, the Fourier coefficients 
$$\hat{g}_{mn}=im\hat{f}_{mn},$$
represent a function $g$ that has discontinuity across $\Gamma^*$, and the above procedure for approximating $\Gamma^*$ can be applied.

\subsection{Error analysis}\hfill

\medskip
We consider the non-smooth bivariate case, where $f$ is a combination of two smooth parts, $f_1$ on $\Omega_1$ and $f_2$ on $\Omega_2$, separated by a smooth curve $\Gamma^*$.
Throughout the paper we approximate $f$ using spline functions. In this section we consider approximations by general approximation spaces. Let $\Pi_1$ be the approximation space for approximating the smooth pieces constituting $f$, and let $\Pi_2$ be the approximation space used for approximating the singularity curve. The following assumption characterize and quantify the assumptions about the function $f$ and its singularity curve $\Gamma^*$ in terms the ability to approximate them using the approximation spaces $\Pi_1,\Pi_2$.

\begin{assumption}\label{eps0}
We assume that $\Pi_1$ and $\Pi_2$ are finite dimensional spaces of dimensions $N_1$ and $N_2$ respectively.
\end{assumption}

\begin{assumption}\label{eps1}
We assume that $f_1$ and $f_2$ are smoothly extendable to $[0,1]^2$ and 
$dist_{[0,1]^2}(f_1, \Pi_1)\le \epsilon,\ \ dist_{[0,1]^2}(f_2, \Pi_1)\le \epsilon.$
\end{assumption}

\begin{assumption}\label{eps2}
For $p\in\Pi_2$, let us denote by $\Gamma_0(p)$ the zero level curve of $p$ within $[0,1]^2$. we assume
there exists $p\in\Pi_2$ such that
$$d_H(\Gamma^*,\Gamma_0(p))\le \delta,$$
where $d_H$ denotes the Hausdorff distance.
\end{assumption}

We look for an approximation $S$ to $f$ which is a combination of two components, $p_1\in\Pi_1$ in $\tilde{\Omega}_1$ and $p_2\in \Pi_1$ in $\tilde{\Omega}_2$, separated by $\Gamma_0(p)$, $p\in\Pi_2$,
such that $(2M+1)^2$ Fourier coefficients of $f$ and $S$  are matched in the least-squares sense:
\begin{equation}\label{LS2DA}
F(p_1,p_2,p)=\sum_{m,n=-M}^M |\hat{f}_{mn}-\hat{S}_{mn}|^2\ \to minimum.
\end{equation}

\begin{assumption}\label{eps3}
Consider the above function $S$ constructed by a triple $(p_1,p_2,\Gamma_0(p))$, $p_1,p_2\in\Pi_1$, $p\in\Pi_2$.
We assume that there is a Lipschitz continuous inverse mapping from  the $(2M+1)^2$ Fourier coefficients of $S$ to the triple $(p_1,p_2,\Gamma_0(p))$:
 \begin{equation}\label{inverse}
\{\hat{S}_{mn}\}_{m,n=-M}^M\ \to\ (p_1,p_2,\Gamma_0(p)).
\end{equation}
\end{assumption}

\begin{remark}\label{enable}
To enable the above property we choose $M$ so that
 \begin{equation}\label{enable}
(2M+1)^2> 2N_1+N_2.
\end{equation}
The topology in the space of triples $(p_1,p_2,\Gamma_0(p))$ is in terms of the maximum norm for the first two components and the Hausdorff distance for the third component. 
\end{remark}

\medskip
\begin{proposition}\label{errorbound}
Let $f_1$, $f_2$, $\Gamma^*$, $\Pi_1$ and $\Pi_2$ satisfy Assumptions 
\ref{eps0}, \ref{eps1}, \ref{eps2} and \ref{eps3}. Then the triple $(p_1^*,p_2^*,p^*)$ minimizing (\ref{LS2DA}) provides the following approximation bounds:
\begin{equation}\label{bound01}
 \|f_1-p_1^*\|_{\infty,\Omega_1^*}\le C_1M\epsilon+C_2M\delta,
\end{equation}
\begin{equation}\label{bound02}
 \|f_2-p_2^*\|_{\infty,\Omega_2^*}\le C_1M\epsilon+C_2M\delta,
\end{equation}
and
\begin{equation}\label{bound1}
d_H(\Gamma^*,\Gamma_0(p^*))\le C_3M\epsilon+C_4M\delta,
\end{equation}
where $\Omega_1^*$ and $\Omega_2^*$  are separated by $\Gamma_0(p^*)$.
\end{proposition}

\begin{proof}
By Assumptions \ref{eps1}, \ref{eps2} it follows that there exists an approximation $S$ defined as above by a triple $(\bar{p}_1,\bar{p}_2,\bar{p})$, such that
\begin{equation}\label{bound01}
 \|f_1-\bar{p}_1\|_{\infty,[0,1]^2}\le \epsilon,
\end{equation}
\begin{equation}\label{bound02}
 \|f_2-\bar{p}_2\|_{\infty,[0,1]^2}\le \epsilon,
\end{equation}
and
\begin{equation}\label{bound03}
d_H(\Gamma^*,\Gamma_0(\bar{p}))\le \delta.
\end{equation}

Building an approximation $\bar{S}$ to $f$ as above by a triple $(\bar{p}_1,\bar{p}_2,\bar{p})$, we can estimate its Fourier coefficients using the above bounds, and it follows that
\begin{equation}\label{bound2}
 |\hat{f}_{mn}-\hat{\bar{S}}_{mn}|\le \epsilon+L\delta, \ \ -M\le m,n\le M.
\end{equation}
Therefore,
\begin{equation}\label{bound3}
min\{F(p_1,p_2,p)\}
\le (2M+1)^2(\epsilon+L\delta)^2.
\end{equation}

Let 
\begin{equation}\label{LS2arg}
\big[p_1^*,p_2^*,p^*\big]=\argmin\bigg\{\sum_{m,n=-M}^M |\hat{f}_{mn}-\hat{S}_{mn}|^2\bigg\}.
\end{equation}
The approximation $S^*$ to $f$  is the combination of the two components, $p_1^*$ in $\Omega_1^*$ and $p_2^*$ in $\Omega_2^*$, where $\Omega_1^*$ and $\Omega_2^*$  are separated by $\Gamma_0(p^*)$.

Using the bound in (\ref{bound3}) it follows that 
\begin{equation}\label{bound*}
|\hat{f}_{mn}-\hat{S}^*_{mn}|\le (2M+1)(\epsilon+L\delta), \ \ -M\le m,n\le M.
\end{equation}
In view of (\ref{bound2}) and (\ref{bound*}) if follows that
\begin{equation}\label{bound*0}
|\hat{\bar{S}}_{mn}-\hat{S}^*_{mn}|\le (2M+2)(\epsilon+L\delta), \ \ -M\le m,n\le M.
\end{equation}

Using Assumption \ref{eps3}, the bound (\ref{bound*0}) implies 
\begin{equation}\label{bound31}
 \|p_1^*-\bar{p}_1\|_{\infty,\Omega_1^*}\le C(2M+2)(\epsilon+L\delta),
\end{equation}
\begin{equation}\label{bound32}
 \|p_2^*-\bar{p}2\|_{\infty,\Omega_2^*}\le C(2M+2)(\epsilon+L\delta),
\end{equation}
and
\begin{equation}\label{bound33}
d_H(\Gamma_0(p^*),\Gamma_0(\bar{p}))\le C(2M+2)(\epsilon+L\delta).
\end{equation}
The approximation result now follows by considering the inequalities (\ref{bound31})-(\ref{bound33}), together with the inequalities (\ref{bound01})-(\ref{bound03}), and applying the triangle inequality.

\end{proof}
\bigskip

\subsubsection{Validity of the approximation assumptions}\hfill

\medskip
Let us check the validity of  Assumptions 
\ref{eps0}, \ref{eps1}, \ref{eps2} and \ref{eps3}
for the approximation tools suggested in Section \ref{Sec2Dns} and used in the above numerical tests.

We assume that $f_1,f_2\in C^m[0,1]^2$, and that $\Gamma^*$ is a $C^m$ curve.
To construct the approximation to $f_1$ and $f_2$ we use the space $\Pi_1$ of $k$th degree tensor-product splines with equidistant knots' distance $d$ in each direction, $k\le m$. The approximation to $\Gamma^*$ is obtained using the approximation space $\Pi_2$ of $\ell$th degree tensor product splines with equidistant knots' distance $h$ in each direction, $\ell\le m$.
$dim(\Pi_1)=(1/d+k-1)^2 \equiv N_d^2 $, $dim(\Pi_2)=(1/h+\ell-1)^2 \equiv N_h^2$, and for both spaces we use the $B$-spline basis functions. Assumptions  \ref{eps1} and \ref{eps2} are fulfilled with $\epsilon=C_1d^k$ and $\delta=C_2h^\ell$.

Assumption \ref{eps3} is more challenging. To define the mapping 
 \begin{equation}\label{inverse1}
\{\hat{S}_{mn}\}_{m,n=-M}^M\ \to\ (p_1,p_2,\Gamma_0(p)),
\end{equation}
we use the same procedure \ref{appproc2D} for defining the approximation to $f$:

We represent $p$ and $p_1,p_2$ using the $B$-spline basis function as in ({\ref{Db}) and (\ref{S12D}), (\ref{S22D}) respectively. Each triple $(p_1,p_2,p)$ defines a piecewise spline approximation $T(x,y)$, and we look for the approximation T(x,y)
such that $(2M+1)^2$ Fourier coefficients of $T$ match the Fourier coefficients $\{\hat{S}_{mn}\}_{m,n=-M}^M$  in the least-squares sense:
\begin{equation}\label{LS2DS}
\big[\{a_{1ij}\}_{i,j=1}^{N_d},\{a_{2ij}\}_{i,j=1}^{N_d},\{b_{ij}\}_{i,j=1}^{N_h}\big]=\argmin\big(\sum_{m,n=-M}^M |\hat{S}_{mn}-\hat{T}_{mn}|^2\big).
\end{equation}
Out of all the possible solutions of the above problem we look for the one with minimial coefficients' norm, i.e., minimizing
\begin{equation}\label{LSaij}
\sum_{i,j=1}^{N_d}a_{1ij}^2+\sum_{i,j=1}^{N_d}a_{2ij}^2.
\end{equation}
Following the procedure \ref{appproc2D} we observe that every step in the procedure is smooth with respect to its input. Possible non-uniqueness in solving the linear system of equations on step (5) is resolved by using the generalized inverse. Therefore, the composition of all the steps is also a smooth function of the input, which implies the validity of Assumption \ref{eps3}.
\bigskip
\vfill\eject
\section{The 3-D case}

\subsection{Numerical Example - The smooth 3-D case}\hfill

\medskip
We consider the test function 
$$f(x,y,z)=(x^2+y^2+z^2-0.5)sin(4(x+y-z)),$$ 
assuming only its Fourier series coefficients are given. We have used only $10^3$ Fourier coefficients and constructed an approximation using 5th-degree tensor product splines with equidistant knots' distance $d=0.1$ in each direction.
 For this case, the matrix $A$ is of size $15^3\times 15^3$, and $cond(A)=1.2\times 10^{22}$. Again, we have employed the iterative refinement algorithm to obtain a high precision solution. The test function is shown in Figure \ref{Fig3Dtestfd01}. The error in the resulting approximation is displayed in Figure \ref{Fig3Derrord01}.
 
 \begin{figure}[!ht]
    \includegraphics[width=4in]{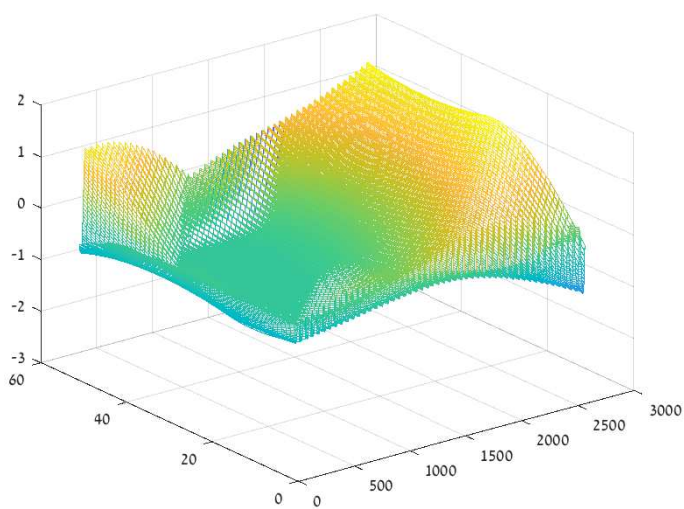}
    \caption{The 3D test function reshaped into 2D.}
    \label{Fig3Dtestfd01}
\end{figure}
 
  \begin{figure}[!ht]
    \includegraphics[width=4in]{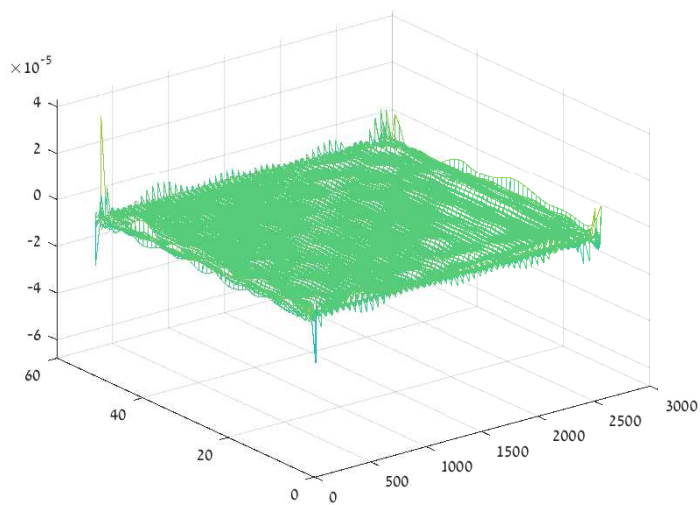}
    \caption{The approximation error graph,  reshaped into 2D.}
    \label{Fig3Derrord01}
\end{figure}
 
\section{Concluding remarks}

The basic crucial assumption behind the presented Fourier acceleration strategy is that the underlying function is piecewise `nice'. That is, piecewisely, the function can be well approximated by a suitable finite set of basis functions. The Fourier series of the function may be given to us as a result of the computational method dictated by the structure of the mathematical problem at hand.
In itself, the Fourier series may not be the best tool for approximating the desired solution, and yet it contains all the information about the requested function. Utilizing this information we can derive high accuracy piecewise approximations to that function.
The simple idea is to make the approximation match the coefficients of the given Fourier series. The suggested method is simple to implement for the approximation of smooth non-periodic functions in any dimension. The case of non-smooth functions is more challenging, and a special strategy is suggested and demonstrated for the univariate and bivariate cases. The paper contains a descriptive graphical presentation of the approximation procedure, together with a fundamentall error analysis.


\bigskip
\begin{thebibliography}{99}
\bibitem{Betankov}Batenkov, Dmitry. "Complete algebraic reconstruction of piecewise-smooth functions from Fourier data." Mathematics of Computation 84.295 (2015): 2329-2350.
\bibitem{Eckhoff1} Eckhoff, K. S. (1993). Accurate and efficient reconstruction of discontinuous functions
from truncated series expansions. Math. Comput. 61, (204), 745–763.
\bibitem{Eckhoff2} Eckhoff, K. S. (1995). Accurate reconstructions of functions of finite regularity from
truncated Fourier series expansions. Math. Comput. 64, (210), 671–690.
\bibitem{Eckhoff3} Eckhoff, K. S. (1998). On a high order numerical method for functions with singularities.
Math. Comput. 67, (223), 1063–1087.
\bibitem{GelbTanner} Gelb, Anne, and Jared Tanner. "Robust reprojection methods for the resolution of the Gibbs phenomenon." Applied and Computational Harmonic Analysis 20.1 (2006): 3-25.
\bibitem{GottliebShu} Gottlieb, D, and C.W. Shu, "On the Gibbs phenomenon and its resolution", SIAM Rev.,39(1977) pp. 644-668.
\bibitem{Gottlieb2011}Gottlieb, Sigal, Jae-Hun Jung, and Saeja Kim. "A review of David Gottlieb’s work on the resolution of the Gibbs phenomenon." Communications in Computational Physics 9.3 (2011): 497-519..
\bibitem{LevinSidi2001}Levin, David, and Avram Sidi. "Extrapolation methods for infinite multiple series and integrals." Journal of Computational Methods in Sciences and Engineering 1.2-3 (2001): 167-184.
\bibitem{Moler} Moler, C. B., ``Iterative refinement in floating point'' Journal ACM, 14(2)(1967) :316–321.
\bibitem{Nersessian} Nersessian, A. N. R. Y., and A. R. N. A. K. Poghosyan. "On a rational linear approximation of Fourier series for smooth functions." Journal of Scientific Computing 26.1 (2006): 111-125.
\bibitem{Sidi1995} Sidi, Avram. "Acceleration of convergence of (generalized) Fourier series by the d-transformation." Ann. Numer. Math 2.1-4 (1995): 381-406.
\bibitem{Tadmor}Tadmor, E., " Filters, mollifiers and the computation of the Gibbs phenomenon.",Acta Numer., 16(2007): 305-378.
\bibitem{Wilk}  [4] J Wilkinson, H., Rounding Errors in Algebraic Processes. Prentice-Hall, 1963.



\end{thebibliography}
\end{document}